\pdfoutput=1 

\documentclass[12pt,a4paper]{amsart}
\usepackage[T1]{fontenc}
\usepackage[top=35mm, bottom=40mm, left=30mm, right=30mm]{geometry}
\usepackage[colorlinks=true,citecolor=blue]{hyperref}
\usepackage{amssymb}
\usepackage{amsthm}
\usepackage{xcolor}
\usepackage{verbatim}

\usepackage[looser]{newtxtext}
\usepackage{newtxmath}

\newtheorem{thmx}{Theorem}

\newtheorem{thmxx}{Theorem}

\newtheorem{thm}{Theorem}[section]
\newtheorem{cor}[thm]{Corollary}
\newtheorem{lem}[thm]{Lemma}
\newtheorem{prop}[thm]{Proposition}

\theoremstyle{definition}

\newtheorem{rem}[thm]{Remark}

\numberwithin{equation}{section}
\DeclareMathOperator{\diam}{diam}
\newcommand{\bbn}{\mathbb{N}}
\newcommand{\eps}{\varepsilon}
\numberwithin{equation}{section}

\title{Stronger versions of sensitivity for minimal group actions}
\author{Jian Li and Yini Yang}
\address{Department of Mathematics, Shantou University,
	Shantou, Guangdong, 515063, P.R. China}
\email{lijian09@mail.ustc.edu.cn}
\email{ynyangchs@foxmail.com}

\begin{document}

\begin{abstract}
	We study several stronger versions of sensitivity for minimal group actions, including $n$-sensitivity, thick $n$-sensitivity and blockily thick $n$-sensitivity, and characterize them by the regionally proximal relation.
\end{abstract}

	\maketitle

\section{Introduction}
By a $\mathbb{Z}$-action (topological) dynamical system, we mean a pair $(X,T)$,
where $X$ is a compact metric space with a metric $d$ and
$T\colon X\to X$ is a homeomorphism.
A dynamical system $(X,T)$ is called \emph{equicontinuous} if the family of maps $\{T^m\colon m\in\mathbb{Z}\}$ are uniformly equicontinuous, that is for every $\varepsilon>0$
there is a $\delta>0$ such that whenever $x,y\in X$ with $d(x,y)<\delta$, $d(T^mx,T^my)<\varepsilon$ for all $m\in \mathbb{Z}$.
Equicontinuous systems are stable in some sense.
The opposite of equicontinuity is sensitive dependence on initial conditions (sensitivity for short), which was introduced
by Ruelle \cite{Ruelle1977}.
A dynamical system $(X,T)$ is called \emph{sensitive}
if there exists a constant $\delta>0$ such that for every opene (open and non-empty) subset $U\subset X$
there exist $x,y\in U$ and
$m\in\bbn$ with $d(T^mx,T^my)>\delta$. The interesting dichotomy theorem proved by Auslander and Yorke \cite{Auslander1980} is as follows:
every minimal system is either equicontinuous or sensitive.

In \cite{X05} Xiong introduced a multi-variate version of sensitivity, called $n$-sensitivity.
Let $n\geq 2$. A dynamical system $(X,T)$ is called \emph{$n$-sensitive} if there exists a constant $\delta>0$ such that
for any opene set $U$, there exist $x_1, x_2, \cdots, x_n\in U$ and $m\in \bbn$ with
\[\min_{1\le i<j\le n}d(T^mx_i,T^mx_j)>\delta.\]
Shao, Ye and Zhang \cite{SYZ,YZ} studied $n$-sensitivity extensively, particularly for minimal systems. Huang, Lu and Ye \cite{Huang2011} finally obtained the structure of $n$-sensitivity for minimal systems.

Let $(X,T)$ be a dynamical system. For $n\geq 2$, $\delta>0$ and an opene subset $U$ of $X$, define
\[N(U,\delta;n):=\Bigl\{m\in\bbn\colon \exists
x_1,x_2,\cdots,x_n\in U \text{ such that }
\min_{1\le i<j\le n}d(T^m x_i,T^m x_j)>\delta\Bigr\}.\]
It is easy to see that $(X,T)$ is $n$-sensitive if and only if there exists a constant $\delta>0$ such that
$N(U,\delta;n)$ is infinite for each opene subset $U$ of $X$.
Following the ideas in \cite{F1981} and \cite{A97},
it is natural to study $N(U,\delta;n)$ via some well-performed subsets of $\mathbb{Z}$, see e.g.~\cite{Moothathu2007}, \cite{HKZ2014}, \cite{Zou17}, \cite{LZ2017} and so on.
We refer the reader to the survey \cite{LY16} for related results.

Recall that a subset $S$ of $\bbn$
is called \emph{thick} if for each $k\in\bbn$, there exists $m_k\in\bbn$ such that $\{m_k,m_k+1,\dotsc,m_k+k\}\subset S$.
Let $n\geq 2$.
A dynamical system $(X,T)$ is called \emph{thickly $n$-sensitive}
if there exists a constant $\delta>0$ such that
$N(U,\delta;n)$ is  thick for each opene subset $U$ of $X$.
In \cite{HKZ2014} Huang, Kolyada and Zhang  obtained an analog of Auslander-Yorke type of dichotomy for minimal system: a minimal system is either thick $2$-sensitive or an almost $1$-to-$1$ extension of its maximal equicontinuous factor.
In \cite{LZ2017} Liu and Zhou showed that a minimal system
is either thickly $n$-sensitive or an almost $m$-to-$1$ extension of its maximal equicontinuous factor for some $m\in\{1,\dotsc,n-1\}$.

A dynamical system $(X,T)$ is called \emph{blockily thickly $n$-sensitive}
if there is a constant $\delta>0$
such that for every $k\in \bbn$ and every opene subset $U$ of $X$,
we can find $x_1,x_2,\dotsc,x_n\in U$ and $m_k\in\bbn$
such that
\[
\{m_k,m_k+1,\dotsc,m_k+k\}\subset
\Bigr\{ m\in\bbn\colon
\min_{1\le i<j\le n}d(T^m x_i,T^m x_j)>\delta\Bigr\}.
\]
In \cite{Y18} Ye and Yu showed that a minimal system is either blockily thickly $2$-sensitive or a proximal extension of its maximal equicontinuous factor.
In \cite{Zou17} Zou showed that a minimal system is
blockily thickly $n$-sensitive if and only if for the factor map to its maximal equicontinuous factor, there exist $n$ pairwise distinct points on fibers such that any two of them form a distal pair.

It is natural to consider sensitivity for group actions or semigroup actions, see e.g.\ \cite{KM08}, \cite{P10} and \cite{LYZ18}.
In this paper, we study $n$-sensitivity, thick $n$-sensitivity and blockily thick $n$-sensitivity for a minimal group action $(X,G)$, where $G$ is a countable discrete group. 
A pair $(x_1,x_2)\in X\times X$ is called \emph{regionally proximal} if for each $\eps> 0$ 
there exist $x_1^{\prime},x_2^{\prime} \in X$ with $d(x_i,x_i^{\prime})<\eps$ for $i=1,2$ and $g\in G$ such that $d(gx_1', gx_2')<\eps$.
Let $Q(X, G)$ be the collection of all regionally proximal pairs of $(X,G)$ 
and $Q[x]=\{y\in X: (x,y)\in Q(X,G)\}$.
We characterize three versions of sensitivity by the regionally proximal relation 
as follows, which are the main results of the paper.

\begin{thmx}\label{thm:main-1}
Let $(X,G)$ be a minimal system. Assume that $n\geq 2$ and $X\times X$ has a dense set of minimal points.
Then $(X,G)$ is $n$-sensitive if and only if
there exists $x\in X$ such that 
$\#Q[x]\geq n$,
where $\#(\,\cdot\,)$ is the cardinality of a set.
\end{thmx}

\begin{thmx}\label{thm:main-2}
Let $(X,G)$ be a minimal system.
Assume that $n\geq 2$ and $X\times X$ has a dense set of minimal points.
Then $(X,G)$ is thickly $n$-sensitive if and only if for any $x\in X$, $\#Q[x]\geq n$.
\end{thmx}

\begin{thmx}\label{thm:main-3}
Let $(X,G)$ be a minimal system. Assume that $n\geq 2$ and $X\times X$ has a dense set of minimal points.
Then $(X,G)$ is blockily thickly $n$-sensitive if and only if 	there exists a point $x\in X$ and pairwise distinct points $x_1,x_2,\dotsc,x_n\in Q[x]$ such that
$(x_1,x_2,\dotsc,x_n)$ is a minimal point for $(X^n,G)$.
\end{thmx}

The paper is organized as follows.
In Section $2$, we recall some definitions and
related results which will be used later. In Section $3$,
we study some properties about the fibers of factor maps between minimal systems, which may be of interest independently.
Theorems~\ref{thm:main-1}, \ref{thm:main-2} and \ref{thm:main-3}
are proved in Sections 4, 5 and 6 respectively.
In section 7, we list a few remarks on the main results.

\section{Preliminaries}

In this section we will recall some basic notions and results which we will need in the following sections.

As usual, the collections of positive integers, integers and real numbers are denoted by $\bbn$, $\mathbb{Z}$ and $\mathbb{R}$ respectively.

Throughout this paper, let $G$ be a countable discrete group with an identity $e$.
We say that a subset $A$ of $G$ is \emph{thick}  if for any finite subset $F$ of $G$ there exists $g\in G$ such that $Fg\subset A$;
and \emph{syndetic}  if there exists a finite subset $K$ of $G$ such that $G=KA$.

Let $X$ be a compact metric space with a metric $d$.
A \emph{$G$-action} on $X$ is a continuous map $\Pi:G\times X\to X$ satisfying
$\Pi(e,x)=x$, $\forall x\in X$
and $\Pi(s,\Pi(g,x))=\Pi(sg,x)$, $\forall x\in X$, $s,g\in G$.
We say that the triple $(X,G,\Pi)$ is a \emph{topological dynamical system}.
For convenience, we will use the pair $(X,G)$ instead of $(X,G,\Pi)$ to denote the topological dynamical system, and $gx:=\Pi(g,x)$ if the map $\Pi$ is unambiguous.
For any $n\in \bbn$, there is a natural $G$-action on the $n$-fold product space $X^n$
as $g(x_1,\dotsc,x_n)=(gx_1,\dotsc,gx_n)$ for every $(x_1,\dotsc,x_n)\in X^n$.

A non-empty closed
$G$-invariant subset $Y\subseteq X$ defines naturally a subsystem $(Y, G)$ of $(X, G)$.
A system $(X,G)$ is called \emph{minimal} if it contains no proper subsystems.
Each point belonging to some minimal subsystem of $(X,G)$ is called a \emph{minimal point}.
By the Zorn's Lemma, every topological dynamical system has a minimal subsystem.

The \emph{orbit} of a point $x\in X$ is the set
$Gx=\{gx:\ g\in G\}$, and the \emph{orbit closure} is $\overline{Gx}$. It is clear that the orbit closure $\overline{Gx}$ is $G$-invariant and $(\overline{Gx},G)$ is a subsystem of $(X,G)$.
Any point with dense orbit is called \emph{transitive}, and in this case the system is called \emph{point transitive}.
It is easy to see that $(X,G)$ is minimal if and only if every point in $X$ is transitive.

A pair $(x_1,x_2)\in X\times X$ is said to be \emph{proximal} if $\inf_{g\in G}d(gx_1, gx_2)=0$,
and \emph{distal} if  $\inf_{g\in G}d(gx_1, gx_2)>0$.
It is easy to see that if $(x_1,x_2)$ is a minimal point in $(X^2,G)$ then either $x_1=x_2$ or $(x_1,x_2)$ is distal.

Let $(X,G)$ and $(Y,G)$ be two dynamical systems. If there is a continuous surjection $\pi: X \to Y$ with $\pi\circ g = g\circ \pi$ for all $g\in G$,
then we say that $\pi$ is a \emph{factor map},
the system $(Y,G)$ is a \emph{factor} of $(X,G)$
or $(X, G)$ is an \emph{extension} of $(Y,G)$.	
If $\pi$ is a homeomorphism, then we say that $\pi$ is a \emph{conjugacy} and
dynamical systems $(X,G)$ and $(Y,G)$ are \emph{conjugate}. Conjugate dynamical systems can
be considered the same from the dynamical point of view.
Let $\pi\colon (X,G)\to (Y,G)$ be a factor map between two dynamical systems and let
\[
R_\pi=\{(x_1,x_2)\in X\times X \colon \pi(x_1)=\pi(x_2)\}.
\]
Then $R_\pi$ is a closed $G$-invariant equivalence relation on $X$ and $Y=X/R_\pi$.
In fact, there exists a one-to-one correspondence between
the collection of factors of $(X,G)$ and the collection of closed $G$-invariant equivalence relations on $X$.
We refer the reader to \cite{A88} for the theory of minimal systems and their extensions.

We say that $(X,G)$ is \emph{equicontinuous} if
for every $\eps>0$ there exists $\delta>0$ such that
for any $x,y\in X$ with $d(x,y)<\delta$, $d(gx,gy)<\eps$ for all $g\in G$.
Equicontinuous system has ``simple'' dynamical behaviors.
For example, if $(X,G)$ is equicontinuous then there exists a compatible metric $\rho$ on $X$ such that the action $G$ on $X$ is isometric, that is $\rho(gx,gy)=\rho(x,y)$ for every $g\in G$ and $x,y\in X$.
Every topological dynamical system $(X,G)$ has a maximal equicontinuous factor $(X_{eq},G)$, that is $(X_{eq},G)$ is equicontinuous
and every equicontinuous factor of $(X,G)$ is also
a factor of $(X_{eq}, G)$.
There is a closed $G$-invariant equivalence relation $S_{eq}$ on $(X,G)$, called the \emph{equicontinuous structure relation}, such that $X/S_{eq}=X_{eq}$.

The equicontinuous structure relation is closely related to the regionally proximal relation.
A pair $(x_1,x_2)\in X\times X$ is called \emph{regionally proximal} if for each $\eps> 0$ 
there exist $x_1^{\prime},x_2^{\prime} \in X$ with $d(x_i,x_i^{\prime})<\eps$ for $i=1,2$ and $g\in G$ such that $d(gx_1', gx_2')<\eps$.
Let $Q(X, G)$ be the collection of all regionally proximal pairs of $(X,G)$. We say that $Q(X, G)$ is the \emph{regionally proximal relation}.  Then $Q(X,G)$ is a reflexive symmetric $G$-invariant closed relation which contains $\Delta_2(X):=\{(x,x)\in X^2\colon x\in X\}$, and it is easy to see that $(X,G)$ is equicontinuous if and only if $Q(X,G)=\Delta_2(X)$.
Let $Q[x]=\{y\in X: (x,y)\in Q(X,G)\}$.
We know that $S_{eq}$ is the smallest closed $G$-invariant
equivalence relation containing $Q(X,G)$, see e.g. \cite[Theorem 9.3]{A88}.
It is interesting to know that when the regionally proximal relation on a minimal system is a closed equivalence relation (and so coincides
with the equicontinuous structure relation).
This happens for many cases, including the acting group $G$ is abelian \cite{V68},
$(X\times X,G)$ has a dense set of minimal points \cite{V75},
and $(X,G)$ admits an invariant probability Borel measure \cite{M78}.

For $n\geq 2$, we can also define the \emph{$n$-th regionally proximal relation} $Q_n(X,G)$ of $(X,G)$ by
$(x_1,x_2\cdots,x_n)\in Q_n(X,G)$
if and only if for any $\eps>0$
there exist $x_1^{\prime},x_2^{\prime}\cdots,x_n^{\prime}\in X$ with $d(x_i,x_i^{\prime})<\eps$ for $i=1,2,\cdots,n$ and $g\in G$ such that  $d(gx_i^{\prime},gx_j^{\prime})<\eps$ for $1\leq i,j\leq n$.
Thus $Q_2(X,G)=Q(X,G)$. Note that if $(x_1,x_2\cdots,x_n)\in Q_n(X,G)$ then $(x_i,x_j)\in Q(X,G)$ for $1\leq i,j\leq n$.
The following result shows that under some conditions the $n$-th regionally proximal relation is a joint of the regionally proximal relation. 

\begin{thm}[{\cite[Theorem 8]{A04}}] \label{03}
Let $(X,G)$ be a minimal system and $n\geq 2$.
Assume that $X\times X$ has a dense set of minimal points.
If  $x_1,x_2,\dotsc,x_n\in X$ satisfy $(x_1,x_i)\in Q(X,G)$ for $1\le i\le n$,
then $(x_1,x_2,\dotsc,x_n)\in Q_n(X,G)$.
\end{thm}

It is clear that if the conclusion of Theorem \ref{03} holds then
$Q(X,G)$ is an equivalence relation, and therefore it coincides with
the equicontinuous structure relation.
Let $\pi\colon (X,G)\to (X_{eq},G)$ be the factor map to its maximal equicontinuous factor.
If $X\times X$ has a dense set of minimal points then $R_\pi=Q(X,G)$.
So we can restate our main results (Theorems~\ref{thm:main-1}, \ref{thm:main-2} and~\ref{thm:main-3}) as follows: 

\begin{thmxx}\label{thm:main-1-x}
	Let $(X,G)$ be a minimal system and $\pi\colon (X,G)\to (X_{eq},G)$ be the factor map to its maximal equicontinuous factor. Assume that $n\geq 2$ and $X\times X$ has a dense set of minimal points.
	Then $(X,G)$ is $n$-sensitive if and only if
	$\sup_{y\in X_{eq}}\#(\pi^{-1}(y))\geq n$.
\end{thmxx}

\begin{thmxx}\label{thm:main-2-x}
    Let $(X,G)$ be a minimal system and $\pi\colon (X,G)\to (X_{eq},G)$ be the factor map to its maximal equicontinuous factor.
	Assume that $n\geq 2$ and $X\times X$ has a dense set of minimal points.
	Then $(X,G)$ is thickly $n$-sensitive if and only if
	$\inf_{y\in X_{eq}}\#(\pi^{-1}(y))\geq n$.
\end{thmxx}

\begin{thmxx}\label{thm:main-3-x}
	Let $(X,G)$ be a minimal system and $\pi\colon (X,G)\to (X_{eq},G)$ be the factor map to its maximal equicontinuous factor. Assume that $n\geq 2$ and $X\times X$ has a dense set of minimal points.
	Then $(X,G)$ is blockily thickly $n$-sensitive if and only if 	there exists a point $y\in X_{eq}$ and pairwise distinct points $x_1,x_2,\dotsc,x_n\in \pi^{-1}(y)$ such that
	$(x_1,x_2,\dotsc,x_n)$ is a minimal point for $(X^n,G)$.
\end{thmxx}

Theorems~\ref{thm:main-1-x}, \ref{thm:main-2-x} and~\ref{thm:main-3-x} show that we can also characterize three versions of sensitivity by the fibers of factor map to the maximal equicontinuous factor. 
One of the advantages of this form is that we should use some properties of factor maps between minimal systems in the proof of main results, which are dealt with in the next section.

\section{Properties of factor maps between minimal systems}
In this section, we study some properties of factor maps between minimal systems, which  may be of interest independently.

Let $X$ and $Y$ be two compact metric space.
We say that a map $\pi\colon X\to Y$ is \emph{semi-open}
if $\pi(U)$ has non-empty interior in $Y$ for every opene subset $U$ of $X$. It is easy to see that $\pi$ is semi-open if and only if $\pi^{-1}(A)$ is a dense $G_\delta$ subset of $X$ for every dense $G_\delta$ subset $A$ of $Y$.

\begin{lem}[{see e.g.\ \cite[Theorem1.15]{A88}}] \label{lem:factor-semi-open}
Let $\pi\colon(X,G)\to (Y,G)$ be a factor map between minimal system.
Then $\pi$ is semi-open.
\end{lem}

We say that a function $\phi\colon X\to \mathbb{R}$ is \emph{upper semi-continuous} if $\limsup_{x\to x_0} \phi(x)\leq \phi(x_0)$ for every $x_0\in X$.
It is easy to see that
$\phi$ is upper semi-continuous if and only if  $\phi^{-1}((-\infty,c))$ is open in $X$ for every $c\in \mathbb{R}$.

\begin{lem}[{see e.g.\ \cite[Lemma 1.28]{F1981}}] \label{lem:usc-continuous-points}
If $X$ is a compact metric space and $\phi\colon X\to \mathbb{R}$ is upper semi-continuous,
then the collection of all continuous points of $\phi$	
is a dense $G_\delta$ subset of $X$.
\end{lem}

The following lemma plays an important role on the study of the cardinality of fibers.
\begin{lem}\label{lem:phi-n-usc}
Let $X$ and $Y$ be two compact metric spaces, and $\pi\colon X\to Y$ be a continuous surjection.
For every $n\geq 2$, define
\begin{align*}
\phi_n\colon Y &\to\mathbb{R}\\
y&\mapsto \sup\Bigl\{\min_{1\le i<j\le n} d(x_i,x_j)\colon x_1,x_2,\dotsc,x_n\in \pi^{-1}(y)\Bigr\}.
\end{align*}
Then
\begin{enumerate}
	\item for any $y\in Y$, $\phi_n(y)>0$ if and only if $\#(\pi^{-1}(y))\geq n$;
	\item $\phi_n$ is upper semicontinuous;
	\item $\{y\in Y\colon \phi_n(y)=0\}$ is a $G_\delta$ subset of $Y$.
\end{enumerate}
\end{lem}
\begin{proof}
(1) is clear.
(2) Fix any $y_0\in Y$. To show that $\phi_n$ is upper semicontinuous,
it suffices to show that
\[\limsup_{y\to y_0}\phi_n(y)\leq \phi_n(y_0).\]
Let $c=\limsup_{y\to y_0}\phi_n(y)$.
Then there exists a sequence $\{y_m\}$ in $Y$ such that
$y_0=\lim_{m\to\infty} y_m$ and $c=\lim_{m\to\infty} \phi_n(y_m)$.
By the definition of $\phi_n$,
there exist sequences $\{a^{(i)}_m\}$ in $X$ with $\pi(a^{(i)}_m)=y_m$ for $i=1,2,\dotsc,n$
such that
\[
\lim_{m\to\infty}\min_{1\leq i<j\leq n}d(a^{(i)}_m,a^{(j)}_m)=c.
\]
As $X$ is compact, without loss of generalization assume that
$a^{(i)}_m \to a^{(i)}$ as $m\to\infty$ for  $i=1,2,\dotsc,n$.
Then $\pi(a^{(i)})=y_0$ for $i=1,2,\dotsc,n$, and
\[
\phi_n(y_0)\geq \min_{1\leq i<j\leq n} d(a^{(i)},a^{(j)})\geq c.
\]
(3) For every $k\in\bbn$, let $Y_k=\phi_n^{-1}([0,\frac{1}{k}))$.
As $\phi_n$ is upper semicontinuous, $Y_k$ is open in $Y$.
Then $\{y\in Y\colon \phi_n(y)=0\}=\bigcap_{k=1}^\infty Y_k$ is a $G_\delta$ subset of $Y$.
\end{proof}

In general, a upper semi-continuous function on a compact space can attain the maximum value, but not the minimum value.
The following result reveals the upper semi-continuous
function as in Lemma~\ref{lem:phi-n-usc} can also attain the minimum value if it is zero, which was inspired by \cite[Lemma 2.4]{DG16} and \cite[Proposition 4.4]{HKZ2014}.

\begin{prop}\label{prop:phi-n-0}
Let $\pi:(X,G)\to (Y, G)$ be a factor map between two minimal systems and $n\geq 2$.
Let $\phi_n$ as in Lemma~\ref{lem:phi-n-usc}.
If $\inf_{y\in Y}\phi_n(y)=0$, then there exists $y_0\in Y$ such that $\phi_n(y_0)=0$.
\end{prop}
\begin{proof}
There exists a sequence $\{y_i\}$ in $Y$ such that
$\lim_{i\to\infty} \phi_n(y_i)=0$.
By Lemma~\ref{lem:usc-continuous-points}, pick a continuous point $y_0\in Y$ of $\phi_n$.
For any $\eps>0$, there exists a open neighborhood $V_\eps $ of $y_0$ such that for any $y\in V_\eps$, $|\phi_n(y)-\phi_n(y_0)|<\eps$.
As $(Y,G)$ is minimal, there exists a finite subset $F$ of $G$
such that $\bigcup_{g\in F} g^{-1} V_\eps =Y$.
Choose a subsequence if necessary, we can assume that
there exists $g_0\in F$ such that $y_i\in g_0^{-1}V_\eps$ for all $i\in\bbn$.
As $\pi^{-1}(g_0y_i)=g_0\pi^{-1}(y_i)$, $g_0$ is continuous and
$\lim_{i\to\infty} \phi_n(y_i)=0$, one has
\begin{align*}
\lim_{i\to\infty} \phi_n(g_0y_i)
&=\lim_{i\to\infty} \sup\Bigl\{ \min_{1\leq i<j\leq n} d(x_i,x_j)\colon x_1,\allowbreak x_2,\dotsc, x_n\in \pi^{-1}(g_0y_i)\Bigr\} \\
&=
\lim_{i\to\infty} \sup\Bigl\{ \min_{1\leq i<j\leq n} d(g_0z_i ,g_0z_j)\colon z_1,\allowbreak z_2,\dotsc, z_n\in \pi^{-1}(y_i)\Bigr\}\\
&=0.
\end{align*}
Then $|\phi_n(y_0)|\le \eps$.
This implies that $\phi_n(y_0)=0$  as $\eps$ is arbitrary.
\end{proof}

We say that a continuous surjection $\pi\colon X\to Y$  is \emph{almost $n$-to-$1$} if there exists a dense $G_\delta$ subset $X_0$ of $X$ such that $\#(\pi^{-1}(\pi(x)))=n$ for every $x\in X_0$.
It is easy to see that if $\pi$ is semi-open then $\pi$ is almost $n$-to-$1$ if and only if there exists a dense $G_\delta$ subset $Y_0$ of $Y$ such that $\#(\pi^{-1}(y))=n$ for every $y\in Y_0$.

\begin{prop}\label{lem:min-card}
Let $\pi:(X,G)\to (Y, G)$ be a factor map between two minimal systems.
If $\pi^{-1}(y_0)$ is a finite set for some $y_0\in Y$
then $\pi$ is almost $n$-to-$1$, where $n=\min_{y\in Y} \#(\pi^{-1}(y))$.
\end{prop}
\begin{proof}
Let $n=\min_{y\in Y} \#(\pi^{-1}(y))$ and $Y_0=\{y\in Y\colon \#(\pi^{-1}(y))=n\}$.
Let $\phi_{n+1}\colon Y\to\mathbb{R}$ as in Lemma~\ref{lem:phi-n-usc}.
Then $Y_0=\{y\in Y \colon \phi_{n+1}(y)=0\}$.
By Proposition~\ref{prop:phi-n-0}, $Y_0$ is the collection of all continuous points of $\phi_{n+1}$.
By Lemma~\ref{lem:usc-continuous-points}, $Y_0$ is a dense $G_\delta$ subset of $Y$.
Thus, $\pi$ is almost $n$-to-$1$.
\end{proof}

\begin{cor}\label{cor:min-card}
	Let $\pi:(X,G)\to (Y, G)$ be a factor map between two minimal systems and $n\in\bbn$.
	Then $\pi$ is almost $n$-to-$1$ if and only if $n=\min_{y\in Y} \#(\pi^{-1}(y))$.
\end{cor}

We have the following observation on the minimal points and distal pairs on fibers.

\begin{lem}\label{lem:factor-minimal-distal}
	Let $\pi:(X,G)\to (Y, G)$ be a factor map between two minimal systems and $n\geq 2$. Then the following conditions are equivalent:
	\begin{enumerate}
		\item there exists $y_0\in Y$ with the property that
		there exist pairwise distinct points $x_1,\dotsc,x_n \in\pi^{-1}(y_0)$ such that $(x_1,\dotsc,x_n)$   is a minimal point in $(X^n,G)$;
		\item for every $y\in Y$,
		there exist pairwise distinct points $x_1,\dotsc,x_n \in\pi^{-1}(y)$ such that $(x_1,\dotsc,x_n)$   is a minimal point in $(X^n,G)$;
		\item there exists $y_0\in Y$ with the property that
		there exist pairwise distinct points $x_1,\dotsc,x_n \in\pi^{-1}(y_0)$ such that $(x_i,x_j)$ is distal for $1\leq i<j\leq n$;
		\item for every $y\in Y$,
		there exist pairwise distinct points $x_1,\dotsc,x_n \in\pi^{-1}(y)$ such that $(x_i,x_j)$ is distal for $1\leq i<j\leq n$.
	\end{enumerate}
\end{lem}

\begin{proof}
	(2)$\Rightarrow$ (4) and (4)$\Rightarrow$(3) are clear.
	
	(1)$\Rightarrow$(2)
	Fix a point $z\in Y$.
	As $(Y,G)$ is minimal,
	there exists a sequence $\{g_i\}$ in $G$ such that  $\lim_{i\to\infty} g_i y_0=z$.
	As $X$ is compact, without loss of generality  assume that $\lim_{i\to\infty}g_ix_j=x_j'\in X$ for $j=1,\dotsc,n$.
	Then
	\begin{align*}
	\pi(x_j')&=\pi(\lim_{i\to\infty}g_ix_j)=
	\lim_{i\to\infty}\pi(g_ix_j)=\lim_{i\to\infty}g_i\pi(x_j)=\lim_{i\to\infty}g_iy_0=z.
	\end{align*}
	Then $x_1',\dotsc,x_n'\in \pi^{-1}(z)$.
	As $(x_1',\dotsc,x_n')$ is in the orbit closure of $(x_1,\dotsc,x_n)$,  $(x_1',\dotsc,x_n')$ is also a minimal point.
	Moreover, as $x_1,\dotsc,x_n$ pairwise distinct, they are pairwise distal.
	So $x_1',\dotsc,x_n'$ are also pairwise distal and pairwise distinct.
	
	(3)$\Rightarrow$(1)
	Let
	\[
	R_\pi^n=\{(x_1,\dotsc,x_n)\in X^n\colon \pi(x_i)=\pi(x_j),\ 1\leq i<j\leq n\}.
	\]
	Then $R_\pi^n$ is a closed $G$-invariant subset set of $X^n$ and
	$(x_1,\dotsc,x_n)\in R_\pi^n$.
	For $1\leq i<j\leq n$, let $\eta_{ij}=\inf_{g\in G}d(gx_i,gx_j)$.
	As each $(x_i,x_j)$ is distal, $\eta_{ij}>0$ and then $\eta=\min_{1\leq i<j\leq n} \eta_{ij}>0$.
	In particular for every $g\in G$,
	$\min_{1\leq i<j\leq n} d(gx_i,gx_j)\geq \eta$.
	Pick a minimal point $(z_1,\dotsc,z_n)$ in the orbit closure of $(x_1,\dotsc,x_n)$. Then $\min_{1\leq i<j\leq n} d(z_i,z_j)\geq \eta$. As $(z_1,\dotsc,z_n) \in R_\pi^n$, there exists $y\in Y$ such that $z_1,\dotsc,z_n\in \pi^{-1}(y)$.
\end{proof}

It is obvious that if the regionally proximal relation is an equivalence relation and $\pi:(X,G)\to (X_{eq}, G)$ is the factor map
to its maximal equicontinuous factor, then the conclusion of Lemma~\ref{lem:factor-minimal-distal} has another version related to the regionally proximal relation, which we conclude as follows. 

\begin{cor}\label{lem:Q-minimal-distal}
	Let $(X,G)$ be a minimal system. 
	Assume that $n\geq 2$ and $X\times X$ has a dense set of minimal points.
	Then the following conditions are equivalent:
	\begin{enumerate}
		\item there exists $x_0\in X$ with the property that
		there exist pairwise distinct points $x_1,\dotsc,x_n \in Q[x_0]$ such that $(x_1,\dotsc,x_n)$   is a minimal point in $(X^n,G)$;
		\item for every $x\in X$,
		there exist pairwise distinct points $x_1,\dotsc,x_n \in Q[x]$ such that $(x_1,\dotsc,x_n)$   is a minimal point in $(X^n,G)$;
		\item there exists $x_0\in X$ with the property that
		there exist pairwise distinct points $x_1,\dotsc,x_n \in Q[x_0]$ such that $(x_i,x_j)$ is distal for $1\leq i<j\leq n$;
		\item for every $x\in X$,
		there exist pairwise distinct points $x_1,\dotsc,x_n \in Q[x]$ such that $(x_i,x_j)$ is distal for $1\leq i<j\leq n$.
	\end{enumerate}
\end{cor}

Let $\pi:(X,G)\to (Y, G)$ be a factor map between two minimal systems.
Define
	\begin{align*}
	r_\pi\colon Y&\to\bbn\cup\{\infty\}\\
	y&\mapsto \sup\{ k\in\bbn\colon \text{ there exist pairwise distint points }  x_1,\dotsc,x_k \in\pi^{-1}(y)\\
	&\qquad \qquad \qquad \ \ \textup{ such that } (x_1,\dotsc,x_k) \text{ is a minimal point in } (X^k,G)
	\}.
	\end{align*}
By Lemma~\ref{lem:factor-minimal-distal},
$r_\pi$ is a constant function.
For convenience, we will also use $r_\pi$ to denote the constant.

\begin{prop}\label{prop:1:minimial-r-pi}
Let $\pi:(X,G)\to (Y, G)$ be a factor map between two minimal systems.
If there exists $y_0\in Y$ such that $\pi^{-1}(y_0)$ is a finite set, then $r_\pi=\min\{ \#(\pi^{-1}(y))\colon y\in Y\}$.
\end{prop}
\begin{proof}
Let $n=\min\{ \#(\pi^{-1}(y))\colon y\in Y\}$.
It is clear that $r_\pi\leq n$.
If $n=1$, then $r_\pi=1$. Now assume that $n\geq 2$.
Let $\phi_n$ as in Lemma \ref{lem:phi-n-usc}.
Then for every $y\in Y$, $\phi_n(y)>0$.
By Proposition \ref{prop:phi-n-0},
$\inf_{y\in Y} \phi_n(y)>0$.
Let $Y_0=\{y\in Y\colon \#(\pi^{-1}(y))=n\}$.
Then $Y_0$ is a dense $G_\delta$ and $G$-invariant set in $Y$ by Corollary~\ref{cor:min-card}.
Choose $y\in Y_0$ and
enumerate $\pi^{-1}(y)$ as $\{x_1,x_2,\dotsc,x_n\}$.
By Lemma~\ref{lem:factor-minimal-distal}, it is sufficient to show that $(x_i,x_j)$ is distal for $1\leq i<j\leq n$.
For every $g\in G$, $\pi(gx_i)=\pi(gx_j)=gy$. Since $gy\in Y_0, \#(\pi^{-1}(gy))=n,$
 $d(gx_i,gx_j)\geq \phi_n(gy)\geq \inf_{y\in Y} \phi_n(y)>0$.
This implies that $\inf_{g\in G} d(gx_i,gx_j)>0$, that is $(x_i,x_j)$ is distal.
\end{proof}

We say that a factor map $\pi:(X,G)\to (Y, G)$ is \emph{proximal}
if $(x_1,x_2)$ is proximal for any $x_1,x_2\in X$ with $\pi(x_1)=\pi(x_2)$. By Lemma \ref{lem:factor-minimal-distal}
and Proposition \ref{prop:1:minimial-r-pi}, we have the following
two corollaries on proximal factor maps.

\begin{cor}\label{cor:proximal-r-pi}
Let $\pi:(X,G)\to (Y, G)$ be a factor map between two minimal systems.
Then $\pi$ is proximal if and only if $r_\pi=1$.
\end{cor}

\begin{cor}
	Let $\pi:(X,G)\to (Y, G)$ be a factor map between two minimal systems.
If $\pi$ is almost 1-to-1, then $\pi$ is proximal.
\end{cor}

\begin{cor} \label{cor:almost-1-1-not-proximal}
Let $\pi:(X,G)\to (Y, G)$ be a factor map between two minimal systems.
If $\pi$ is  proximal but not almost $1$-to-$1$, then for every $y\in Y$, $\pi^{-1}(y)$ is an infinite set.	
\end{cor}

\section{$N$-sensitivity}
In this section we introduce the concept of $n$-sensitivity for group actions and prove Theorem~\ref{thm:main-1}. 

Let $(X,G)$ be a topological dynamical system and $n\geq 2$.
We say that $(X,G)$ is \emph{$n$-sensitive}
if there exists $\delta>0$ such that for every opene subset $U$ of $X$,
there exist pairwise distinct points $x_1,x_2,\dotsc,x_n\in U$ and
$g\in G$ such that $\min_{1\le i<j\le n}d(gx_i,gx_j)>\delta$.

Now we begin to prove Theorem~\ref{thm:main-1}.
\begin{proof}[Proof of Theorem~\ref{thm:main-1}]
	($\Rightarrow$) Assume $(X,G)$ is $n$-sensitive. Pick a point $x \in X$. Let $U_m$ be a sequence of open neighborhoods of $x$ with $\diam (U_m)<\frac{1}{m}$ for $m\in\bbn$. Then there is $\delta>0$ such that for each $m\in \bbn$, there exist $x_1^m,\ldots,x_n^m\in U_m$ and $s_m\in G$ satisfying $d(s_mx_i^m,s_mx_j^m)>\delta$ for $1\leq i<j\leq n$. Without loss of generality we assume that $s_mx_i^m\to x_i\ (m\to\infty)$ for $i=1,\ldots,n$.
	Then $d(x_i,x_j)\geq \delta$ for $1\leq i<j\leq n$, which shows that $x_1,\ldots,x_n$ are pairwise distinct.
	Then it is sufficient to show that$(x_1,\ldots,x_n)\in Q_n(X,G)$. For any $\eps>0$, take $V_i=B(x_i,\eps)$, $i=1,\dotsc, n$. There exists $m\in\bbn$ such that $s_mx_i^m\in V_i$, $i=1,\ldots,n$ and $\frac{1}{m}<\eps$. It is clear that for any $1\leq i<j\leq n$,
	\[
	d(s_m^{-1}(s_mx_i^m), s_m^{-1}(s_mx_j^m))=d(x_i^m, x_j^m)< \diam (U_m)<\frac{1}{m}<\eps.
	\]
	Thus $(x_1,\dotsc, x_n)\in Q_n(X,G)$.
	
	($\Leftarrow$)
	Since there exists $x\in X$ such that $\#Q[x]\geq n$, there exist
	$(x_1,x_2,\dotsc,x_n)\in Q_n(X,G)\setminus \Delta^{(n)}(X)$ by Theorem \ref{03},  where $\Delta^{(n)}(X)=\{(x_1,x_2,\dotsc,x_n)\in X^n\colon
	\exists i<j \text{ s.t. } x_i=x_j\}$.
	Let $\delta=\frac{1}{3}\min_{1\le i<j\le n}d(x_i,x_j)>0$.
	For any $m\in\bbn$, there are points
	$y_i^m\in B(x_i,\frac{1}{m})$ and $s_m\in G$ such that $d(s_my_i^m,s_my_j^m)<\frac{1}{m}$ for $1\leq i<j\leq n$.
	As $X$ is compact, without loss of generality
	assume that $\lim_{m\to\infty}s_my_i^m=x\in X$ for any $1\le i\le n$.
	Since $(X,G)$ is minimal, $x$ has a dense orbit.
	For any opene subset $U$ of $X$, there is $t\in G$ such that $tx\in U$. So there exists $m\in\bbn$ such that $ts_my_i^m\in U$ for all $1\le i\le n$ and $\frac{1}{m}<\delta$.
	Since $(ts_m)^{-1} (ts_my_i^m)=y_i^m\in B(x_i,\frac{1}{m}),\ i=1,\dots,n$,
	one has 
	\[
	d((ts_m)^{-1} (ts_my_i^m), (ts_m)^{-1} (ts_my_j^m))=d(y_i^m, y_j^m)>\delta, \ \text{for any}\ 1\leq i<j\leq n.
	\]	
	Then $(X,G)$ is $n$-sensitive.
\end{proof}

We have the following direct consequences of Theorems~\ref{thm:main-1} and~\ref{thm:main-1-x}.
\begin{cor}
	Let $(X,G)$ be a minimal dynamical system. Assume that $n\geq 2$ and $X\times X$ has a dense set of minimal points.
	Then the following are equivalent:
	\begin{enumerate}
		\item $(X,G)$ is $n$-sensitive but not $(n+1)$-sensitive;
		\item there exists $x_0\in X$ such that $\#Q[x_0]=n$ and for any $x\in X$, $\#Q[x]<n+1$;
		\item $\max_{y\in X_{eq}} \#(\pi^{-1}(y))=n$,
		where $\pi\colon (X,G)\to (X_{eq},G)$ is the factor map to its maximal equicontinuous factor. 
	\end{enumerate}
\end{cor}

\begin{cor}
	Let $(X,G)$ be a minimal dynamical system. Assume that  $X\times X$ has a dense set of minimal points.
	Then the following are equivalent:
	\begin{enumerate}
		\item $(X,G)$ is $n$-sensitive for every $n\geq 2$;
		\item for every $n\geq 2$, there exists $x_0\in X$ such that $\#Q[x_0]\geq n$;
		\item for every $n\geq 2$, $\sup_{y\in X_{eq}} \#(\pi^{-1}(y))\geq n$,
		where $\pi\colon (X,G)\to (X_{eq},G)$ is the factor map to its maximal equicontinuous factor. 
	\end{enumerate}
\end{cor}

\section{Thick $n$-sensitivity}
In this section we introduce the concept of thick $n$-sensitivity for group actions and prove Theorem~\ref{thm:main-2-x}.
We also give some consequences of Theorem~\ref{thm:main-2-x}.

Let $(X,G)$ be a topological dynamical system and $n\geq 2$.
We say that $(X,G)$ is \emph{thickly $n$-sensitive}
if there is a constant $\delta>0$ such that for any opene subset $U$ of $X$, \[N(U,\delta;n):=\{g\in G: \exists
x_1,x_2,\cdots,x_n\in U \text{ such that}
\min_{1\le i<j\le n}d(gx_i,gx_j)>\delta\}\]
is a thick set.

\begin{prop}\label{prop:r-1-not-thk}
	Let $(X,G)$ be a minimal system and $\pi \colon  (X,G) \to (X_{eq}, G)$ the factor map. If $n\ge 2$ and $\min_{y\in X_{eq}} \#(\pi^{-1}(y))<n$, then $(X,G)$ is not thickly $n$-sensitive.
\end{prop}

\begin{proof}
	We prove this result by contradiction.
	Assume that $(X, G)$ is thickly $n$-sensitive with a sensitive constant $\delta$.
	Let $l=\min_{y\in X_{eq}} \#(\pi^{-1}(y))$.
	There exists some $y_0\in X_{eq}$
	such that $\pi^{- 1} (y_0)= \{x_1,x_2,\dotsc,x_l\}$.
	As $(X_{eq},G)$ is equicontinuous, choose a compatible metric $\rho$ on $X_{eq}$ such that $\rho (gy_1,gy_2)=\rho (y_1,y_2)$ for any $g\in G$ and $y_1,y_2\in X_{eq}$.
	We take an open neighborhood $W_i$ of $x_i$ for $1\le i\le l$ such that $\diam (W_i)< \delta$. 
	
	We claim that there exists an open neighborhood $V$ of $y_0$ such that $\pi^{- 1} (V)\subset \cup_{i=1}^lW_i$.
	Otherwise there exists $x^m\in X\setminus \cup_{i=1}^lW_i$ with $\pi(x^m)=y^m$ and $y^m\rightarrow y_0$.
	Without loss of generality we assume  $x^m\rightarrow x_0$, 
	then $x_0\in \pi^{-1}(y_0)\cap (X\setminus \cup_{i=1}^lW_i)$, which contradict
	to the fact $\pi^{-1}(y_0)=\{x_1,x_2,\cdots,x_l\}\subset \cup_{i=1}^lW_i$.
	
	Let $\eps> 0$ be small enough such that $B_{\rho}(y_0,2\eps)\subset V$.
	Take $U= \pi^{- 1} (B_{\rho}(y_0,\eps))$ and set $S=  \{g\in G\colon gy_0\in B_{\rho}(y_0,\eps)\}$.
	Note that $S$ is syndetic as $y_0$ is a minimal point.
	For any $y\in B_{\rho}(y_0,\eps)$ and $g\in S$,
	$\rho(gy_0, gy)=\rho(y_0, y)< \eps$ and then 
	\[\rho(y_0, g y) \leq \rho(y_0,gy_0)+\rho(gy_0,gy)< 2 \eps.\]
	This gives $g B_{\rho}(y_0,\eps)\subset V$, and hence
	\[g U= g \pi^{- 1} (B_{\rho}(y_0,\eps))= \pi^{- 1} (g B_{\rho}(y_0,\eps))\subset \pi^{- 1} (V)\subset \cup_{i=1}^lW_i.\]
	Since $l\le n-1$, for $g\in S$ and 
	any $z_1,z_2,\cdots,z_n\in U$,
	there exist $1\le s< t\le n$ and $i\in\{1,\dotsc,l\}$ such that $gz_s,gz_t\in W_i$. Then  $d(gz_s,gz_t)<\delta$ as $\diam (W_i)< \delta$.
	This implies that $N(U,\delta;n)\cap S=\emptyset$, and then $N(U,\delta;n)$ is not thick,
	which contradicts to the assumption that $\delta$ is a thickly $n$-sensitive constant.	
\end{proof}
 
\begin{proof}[Proof of Theorem \ref{thm:main-2-x}]
($\Rightarrow$) It follows from Proposition \ref{prop:r-1-not-thk}
	
($\Leftarrow$) Let $\delta=\frac{1}{2}\inf_{y\in X_{eq}}\phi _{n}(y)$.
By Proposition~\ref{prop:phi-n-0}, $\delta>0$.
Since $\pi$ is an extension between $(X,G)$ and its maximal equicontinuous factor and $X\times X$ has a dense set of minimal points,
$R_\pi = Q(X,G)$ and 
$R_\pi^n = Q_{n}(X,G)$ by Theorem \ref{03}.
Fix a finite subset $F_m=\{g_1,g_2,\dotsc,g_m\}$ of $G$ and a point $y\in X_{eq}$.
Since $\delta=\frac{1}{2}\inf_{y\in X_{eq}}\phi _{n}(y)>0$,
for each $i=1,2,\dotsc,m$,
there exist $u_i^1,\dotsc,u_i^n\in \pi^{-1}(g_{i}y)$ such that
\[
\min_{1\le j\neq k\le n}d(u_i^j,u_i^k)\geq 2\delta.
\]
Taking
$x_i^j=g_{i}^{-1}u_i^j$ for $i=1,\dotsc,m$ and $j=1,\dotsc,n$, 
one has	$\pi(x_i^j)=y$ and
\[
\min_{1\le j\neq t\le n}d(g_{i}x_{i}^j,g_{i}x_{i}^t)\geq 2\delta.
\]	
Then 
\[(x_i^j)_{1\le i\le m, 1\le j\le n}\in R_{\pi}^{nm} = Q_{nm}(X,G).\]
It is clear that when we fix $i\in\{1,\dotsc,m\}$, $(x_i^j)_{1\le j\le n}$ are pairwise distinct.
For any $k\in\bbn$, there are points
$y_i^{j,k}\in B(x_i^j,1/k)$ and $s_k\in G$ such that $d(s_ky_i^{j,k},s_ky_l^{t,k})<1/k$ for $1\leq i<l\leq m$, $1\leq j<t\leq n$ and for any $1\leq i\leq m$, $(y_i^{j,k})_{1\le j\le n}$ are pairwise distinct when $k$ is large enough.
As $X$ is compact, without loss of generality
assume that $\lim_{k\to\infty}s_ky_i^{j,k}=x\in X$ for any $1\leq i\leq m$, $1\leq j\leq n$.
Since $(X,G)$ is minimal, $x$ has a dense orbit.
For any opene subset $U$ of $X$, there is $t\in G$ such that $tx\in U$. 
Fix $i\in\{1,\dotsc,m\}$ and choose a large enough $k$ such that 
$ts_ky_i^{j,k}\in U$ for all $1\leq j\leq n$ and 
\[
\min_{1\le j\neq t\le n}d(g_{i}B(x_i^j,1/k),g_{i}B(x_i^t,1/k))>\delta\]
Since $(ts_k)^{-1} (ts_ky_i^{j,k})=y_i^{j,k}\in B(x_i^j,1/k),\ 1\leq j\leq n$,

\[\min_{1\le j\neq t\le n}d(g(ts_k)^{-1} (ts_ky_i^{j,k}), g(ts_k)^{-1} (ts_ky_i^{t,k}))>\delta,\ \forall g\in F_m.\]
It is clear that $(ts_ky_i^{j,k})_{1\leq j\leq n}$ are pairwise distinct and we obtain our result.
\end{proof}

We have the following consequences of Theorems~\ref{thm:main-2} and~\ref{thm:main-2-x}.

\begin{cor}\label{cor:thick-n-n+1}
	Let $(X,G)$ be a minimal system.
	Assume that $n\geq 2$ and $X\times X$ has a dense set of minimal points. 
	Then the following are equivalent:
	\begin{enumerate}
		\item $(X,G)$ is thickly $n$-sensitive but not thickly $(n+1)$-sensitive;
		\item there exists $x_0\in X$ such that $\#Q[x_0]=n$ and $\#Q[x]\geq n$ for any $x\in X$;
		\item $\min_{y\in X_{eq}}\#(\pi^{-1}(y))= n$,
		where $\pi\colon (X,G)\to (X_{eq},G)$ is the factor to its maximal equicontinuous factor.
	\end{enumerate}
\end{cor}

\begin{cor}
	Let $(X,G)$ be a minimal dynamical system. Assume that $X\times X$ has a dense set of minimal points.
	Then the following are equivalent:
	\begin{enumerate}
		\item $(X,G)$ is thickly $n$-sensitive for every $n\geq 2$;
		\item $Q[x]$ is an infinite set for every $x\in X$;
		\item  $\pi^{-1}(y_0)$ is an infinite set for every $y_0\in X_{eq}$, where $\pi\colon (X,G)\to (X_{eq},G)$ is the factor to its maximal equicontinuous factor.
	\end{enumerate}
\end{cor}

\begin{cor}
	Let $(X,G)$ be a minimal system and $\pi\colon (X,G)\to (X_{eq},G)$ be the factor map to its maximal equicontinuous factor.
	Assume that  $X\times X$ has a dense set of minimal points.
	If $\pi$ is proximal but not almost  $1$-to-$1$, then $(X,G)$ is thickly $n$-sensitive for all $n\geq 2$.
\end{cor}

\begin{cor}
	Let $(X,G)$ be a minimal system and $\pi\colon (X,G)\to (X_{eq},G)$ be the factor map to its maximal equicontinuous factor.
	Assume that  $X\times X$ has a dense set of minimal points.
	Then either $(X,G)$ is thickly $2$-sensitive or $\pi$ is almost $1$-to-$1$.
\end{cor}
\begin{proof}
	Assume that $(X,G)$ is not thickly $2$-sensitive, by Theorem~\ref{thm:main-2-x},
	there exists $y_0\in X_{eq}$ such that $\#\pi^{-1}(y_0)<2$,  thus
	$\min_{y\in X_{eq}}\#\pi^{-1}(y)=1$, by Corollary~\ref{cor:min-card},
	$\pi$ is almost $1$-to-$1$.
	
	Now we assume that $(X,G)$ is thickly $2$-sensitive, by Theorem~\ref{thm:main-2-x},
	for any $y\in X_{eq}$, $\pi^{-1}(y)\ge 2$, which implies that $\pi$ is not almost $1$-to-$1$.
\end{proof}

\section{Blockily thick $n$-sensitivity}
The aim of this section is to prove Theorem \ref{thm:main-3-x}.
Let $n\geq 2$. A  topological dynamical system $(X,G)$ is called
\emph{blockily thickly $n$-sensitive}
if there is a constant $\delta>0$, such that for any finite subset $F$ of $G$ and any opene set $U$, there exist $x_1, x_2, \dotsc,x_n\in U$ and $h\in G$
such that
\[\min_{1\le i<j\le n} d(ghx_i, ghx_j)> \delta,\ \forall g\in F.\]

\begin{prop}\label{prop:block-minimal-point}
	Let $(X,G)$ be a minimal system and $\pi\colon (X,G)\to (X_{eq},G)$ be the factor map to its maximal equicontinuous factor. Assume that $n\geq 2$.
	If $(X,G)$ is blockily thickly $n$-sensitive, then	there exists a point $y\in X_{eq}$ and pairwise distinct points $x_1,x_2,\dotsc,x_n\in \pi^{-1}(y)$ such that
	$(x_1,x_2,\dotsc,x_n)$ is a minimal point for $(X^n,G)$.	
\end{prop}
\begin{proof}
	
	Since $(X_{eq}, G)$ is equicontinuous,
	we can choose a $G$-invariant metric on $\rho$, i.e.,
	for any $(u,v)\in X_{eq}\times X_{eq}$ and any $g\in G$,
	$\rho(gu,gv)=\rho(u,v)$.
	Let $\eps_k>0$ with $\eps_k\to 0$.  Then for each $k\in\bbn$, there is $0<\tau_k<\eps_k$ such that
	if $w_1,w_2\in X$ with $d(w_1,w_2)<2\tau_k$ then $\rho(\pi(w_1),\pi(w_2))<\eps_k$.
	
	Since $G$ is a countable discrete group, there exists an increasing sequence of finite subsets $F_m=\{g_1,g_2,\dotsc,g_m\}$ of $G$ such that
	$\cup_{m=1}^\infty F_m=G$.
	Fix $x\in X$ and let $U_k=B(x,\tau_k)$.
	By the assumption $(X,G)$ is blockily thickly $n$-sensitive with sensitive constant $\delta>0$,
	thus for any $k\in \bbn$, there are $u_{k}^i\in U_k,i=1,2,\dotsc,n$  and $h_{k} \in G$
	such that for any $g\in F_k$, 
	\[
	\min_{1\le i<j\le n}d(gh_ku_{k}^i,gh_ku_{k}^j)\geq \delta.
	\]	
	Without loss of generality we assume that $h_{k}u_{k}^i\to u^i$
	when $k\to \infty$. Since $\cup_{m=1}^\infty F_m=G$,
	$\min_{1\le i<j\le n}d(gu^i,gu^j)\geq \delta$ for any $g\in G$.
	So $(u^i, u^j)$ is a distal pair.
	
	Now we show that $(u^i,u^j)\in R_{\pi}$. Since $u_{k}^i, u_{k}^j\in U_k$,
	we have $\rho(\pi(u_{k}^i),\pi(u_{k}^j))<\eps_k$.
	Then $\rho(g\pi(u_{k}^i),g\pi(u_{k}^j))<\eps_k$ for any $g\in G$. Since $h_{k}u_{k}^i\to u^i$ for $i=1,2,\cdots,n$ $$\rho(\pi(u^i),\pi(u^j))\le \eps_k$$ for each $k\in\bbn$.
	This implies that $\rho(\pi(u^i),\pi(u^j))=0$, i.e., $\pi(u^i)=\pi(u^j)$.
	By Lemma~\ref{lem:factor-minimal-distal}, there exists a point $y\in X_{eq}$ and pairwise distinct points $x_1,x_2,\dotsc,x_n\in \pi^{-1}(y)$ such that
	$(x_1,x_2,\dotsc,x_n)$ is a minimal point for $(X^n,G)$.
\end{proof}

\begin{proof}[Proof of Theorem \ref{thm:main-3-x}]
$(\Rightarrow)$ It follows from Proposition~\ref{prop:block-minimal-point}.

$(\Leftarrow)$
By Lemma~\ref{lem:factor-minimal-distal},
there are pairwise distinct $x_1,\dotsc,x_n\in X$ such that $\pi(x_i)=\pi(x_j)$
and $(x_i,x_j)$ is a distal pair $1\le i<j \le n$. Since $X\times X$ has a dense set of minimal points, $(x_i,x_j)\in Q(X,G)$ and thus $(x_1,\cdots,x_n)\in Q_{n}(X,G)$ by Theorem~\ref{03}.
Let 
\[\delta=\frac{1}{2}\min_{1\le i< j\le n}\inf_{g\in G} d(gx_i,gx_j)>0.\]
For any $m\in\bbn$, there are points
$y_i^m\in B(x_i,1/m)$ and $s_m\in G$ such that $d(s_my_i^m,s_my_j^m)<1/m$ for $1\leq i<j\leq n$.
As $X$ is compact, without loss of generality
assume that $\lim_{m\to\infty}s_my_i^m=x$ for any $1\le i\le n$.
Since $(X,G)$ is minimal, $x$ has a dense orbit.
For any opene subset $U$ of $X$, there is $t\in G$ such that $tx\in U$. 
For any finite set $F\subset G$, 
choose a large enough $m_F$ such that 
\[\min_{1\le i<j\le n}d(gB(x_i,\frac{1}{m_F}),gB(x_j,\frac{1}{m_F}))>\delta,\ \forall g\in F\]
and 
$ts_{m_F}y_i^{m_F}\in U$ for all $1\le i\le n$.
Since $(ts_{m_F})^{-1} (ts_{m_F}y_i^{m_F})=y_i^{m_F}\in B(x_i,1/{m_F}),\ i=1,\dots,n$,
\[\min_{1\le i< j\le n}d(g(ts_{m_F})^{-1} (ts_{m_F}y_i^{m_F}),g(ts_{m_F})^{-1} (ts_{m_F}y_j^{m_F}))>\delta,\ \forall g\in F.\]
It is clear that $(ts_{m_F}y_i)_{1\leq i\leq n}$ are pairwise distinct and we obtain the result.
\end{proof}

We have the following direct consequences of Theorems~\ref{thm:main-3} and \ref{thm:main-3-x}.

\begin{cor}\label{cor:block-n-n+1}
	Let $(X,G)$ be a minimal system and $\pi\colon (X,G)\to (X_{eq},G)$ be the factor to its maximal equicontinuous factor.
	Assume that $n\geq 2$ and $X\times X$ has a dense set minimal points.
	Then $(X,G)$ is blockily thickly $n$-sensitive but not blockily thickly  $(n+1)$-sensitive
	if and only if $r_{\pi}=n$.
\end{cor}
\begin{proof}
	It follows from the definition of $r_{\pi}$ and Theorem~\ref{thm:main-3-x}.
\end{proof}

\begin{cor}
	Let $(X,G)$ be a minimal system.
	Assume that $n\geq 2$ and $X\times X$ has a dense set of minimal points.
	If $(X,G)$ is thickly $n$-sensitive but not thickly  $(n+1)$-sensitive, then $(X,G)$ is also blockily thickly $n$-sensitive but not blockily thickly  $(n+1)$-sensitive.
\end{cor}

\begin{proof}
	By Corollary~\ref{cor:thick-n-n+1},
	one has $\min_{y\in X_{eq}}\#\pi^{-1}(y)=n$. Now by Proposition~\ref{prop:1:minimial-r-pi}, $r_{\pi}=n$. Finally by Corollary~\ref{cor:block-n-n+1},
	$(X,G)$ is blockily thickly $n$-sensitive but not blockily thickly  $(n+1)$-sensitive.
\end{proof}

\begin{cor}\label{cor:block-thick-sensitive-dichotomy}
	Let $(X,G)$ be a minimal system and $\pi\colon (X,G)\to (X_{eq},G)$ be the factor to its maximal equicontinuous factor.
	Assume that  $X\times X$ has a dense set minimal points.
	Then either $(X,G)$ is blockily thickly $2$-sensitive or
	$\pi$ is proximal.	
\end{cor}
\begin{proof}
	It follows from  Corollary~\ref{cor:proximal-r-pi} and Theorem~\ref{thm:main-3-x}. 
\end{proof}

\begin{rem}\label{rek:proximal-not-1-1}
	Let $(X,G)$ be a minimal system and $\pi\colon (X,G)\to (X_{eq},G)$ be the factor to its maximal equicontinuous factor.
	Assume that $n\geq 2$ and $X\times X$ has a dense set minimal points.	
	If $\pi$ is proximal but not almost $1$-to-$1$, 
	then $(X,G)$ is not blockily thickly $2$-sensitive, 
	but it is thickly $n$-sensitive for all $n\geq 2$.
\end{rem} 

\section{Final remarks}

\begin{rem}
In this paper we assume that
$G$ is a countable discrete group for simplicity.
In fact, Theorem~\ref{thm:main-1} holds for general Hausdorff topological group actions without changing of the proof. For Theorems \ref{thm:main-2-x} and \ref{thm:main-3-x}, we need the acting group $G$ to be $\sigma$-compact and modify the proofs slightly. 
Recall that a Hausdorff topological group $G$ is \emph{$\sigma$-compact} if there exists a sequence $\{F_i\}$  of compact subsets of $G$ such that $\cup_{i=1}^\infty F_i=G$. 
And in this case we say that a subset $A$ of $G$ is \emph{thick}  if for any compact set $F$ of $G$ there exists $g\in G$ such that $Fg\subset A$;
and \emph{syndetic}  if there exists a compact set $K$ of $G$ such that $G=KA$.
\end{rem}

\begin{rem}
In the statement of Theorem~\ref{03}, we state a condition that $X\times X$ has a dense set of minimal points.
In fact the origin condition in \cite[Theorem 8]{A04} is a group theoretic condition on Ellis group of the system, which is weaker than the condition stated in Theorem~\ref{03} by \cite[Remark 1.13 (2)]{AEE95}.

Following \cite{A01}, a minimal system $(X,G)$  is said to satisfy the \emph{local Bronstein condition} if whenever $(x,y)$ is a minimal point in $Q(X,G)$, there is a sequence $\{(x_n,y_n)\}$ of minimal points in $X\times X$ and a sequence $\{g_n\}$ in $G$ such that
$(x_n,y_n)\to (x,y)$ and $d(g_nx_n,g_ny_n)\to 0$ as $n\to\infty$.
It is clear that if $X\times X$ has a dense set of minimal points
then $(X,G)$ satisfies the local Bronstein condition.
According \cite[Theorem 2]{A01}, Theorem~\ref{03} also holds by replacing the condition  $X\times X$ having a dense set of minimal points by $(X,G)$ satisfying the local Bronstein condition.

It is shown in \cite{M78} that if a minimal system $(X,G)$ admits an invariant probability Borel measure then the regional proximal relation is an equivalence relation.
It is interesting to know whether one can replace the conidtion $X\times X$ having a dense set of minimal points in Theorem~\ref{03} by $(X,G)$ admiting an invariant probability Borel measure.
\end{rem}

\begin{rem}
It is natural to consider equicontinuity and sensitivity in the relative cases.
Let $\pi\colon (X,G)\to (Y,G)$ be a factor map.
Recall that we say that $\pi$ is \emph{relatively equicontinuous}
if for every $\eps>0$ there exists $\delta>0$ such that for any $x,y\in X$ with $d(x,y)<\delta$ and $\pi(x)=\pi(y)$,
$d(gx,gy)<\eps$ for all $g\in G$;
and \emph{relatively sensitive} if 
there exists a constant $\delta>0$ such that for any opene subset $U$ of $X$, there exist $x,y\in U$ and $g\in G$ such that 
$\pi(x)=\pi(y)$ and $d(gx,gy)>\delta$. 
In \cite{YZ20}, Yu and Zhou proved a relative version of Corollary \ref{cor:block-thick-sensitive-dichotomy} and a weak version of one side result of Theorem ~\ref{thm:main-1} for $\mathbb{Z}$-actions.
Recently, in \cite{D20} Dai studied the structure theorems of minimal semigroup 
actions and generalized Yu-Zhou's result to minimal semigroup 
actions.
It is interesting how to obtain relative versions of our main results Theorems \ref{thm:main-1}, \ref{thm:main-2} and~\ref{thm:main-3} and the key point is the relative version of Theorem~\ref{03}.
\end{rem}

\subsection*{Acknowledgments.}
The authors were supported by  NNSF of China
(11771264, 11871188) and NSF of Guangdong Province (2018B030306024). The authors would like to thank Professors Xiongping Dai and Tao Yu for helpful suggestions.
The authors  would also like to thank the anonymous referee for the careful reading and
helpful suggestions.


\begin{thebibliography}{99}
	
\bibitem{A97} E. Akin, \emph{Recurrence in topological dynamics, Furstenberg families and Ellis actions}, The University Series in Mathematics. Plenum Press, New York, 1997.

\bibitem{A88} 	
J. Auslander, \textit{Minimal flows and their extensions}, North-Holland Mathematics Studies, 153. Mathematical Notes, 122. North-Holland Publishing Co., Amsterdam, 1988.

\bibitem{A01} J. Auslander, \textit{Minimal flows with a closed proximal cell}, Ergodic Theory Dynam. Systems \textbf{21} (2001), no. 3, 641--645.

\bibitem{A04} J. Auslander, \textit{A group theoretic condition in topological dynamics}, Proceedings of the 18th Summer Conference on Topology and its Applications. Topology Proc. \textbf{28} (2004), no. 2, 327--334.

\bibitem{AEE95} J. Auslander, D. Ellis and R. Ellis,  \textit{The regionally proximal relation}, Trans. Amer. Math. Soc. \textbf{347} (1995), no. 6, 2139--2146.

\bibitem{Auslander1980}
J. Auslander and J. A. Yorke, \textit{Interval maps, factors of maps, and chaos}, T\^ohoku Math. J. (2) \textbf{32} (1980), no.~2, 177--188.

\bibitem{D20} X. Dai, \textit{On the structures of homomorphisms in minimal topological dynamics}, preprint, 2020.

\bibitem{DG16} T. Downarowicz and E. Glasner, \textit{Isomorphic extensions and applications}, Topol. Methods Nonlinear Anal. \textbf{48} (2016), no. 1, 321--338.

\bibitem{F1981} H. Furstenberg, \textit{Recurrence in ergodic theory and combinatorial number theory}, M. B. Porter
Lectures. Princeton University Press, Princeton, N.J., 1981.

\bibitem{HKZ2014}
W. Huang, S. Kolyada and G. Zhang, \textit{Analogues of Auslander-Yorke theorems for multi-sensitivity}, Ergodic Theory Dynam. Systems \textbf{38} (2018), no. 2,  651--665.

\bibitem{Huang2011}
W. Huang, P. Lu, and X. Ye, \textit{Measure-theoretical sensitivity and
equicontinuity}, Israel J. Math. \textbf{183} (2011), 233--283.


\bibitem{KM08} E. Kontorovich and M. Megrelishvili, \textit{A note on sensitivity of semigroup actions}, Semigroup Forum \textbf{76} (2008), no.~1, 133--141.

\bibitem{LY16} J. Li and X. Ye, \textit{Recent development of chaos theory in topological dynamics}, Acta Math.
Sin. (Engl. Ser.) \textbf{32} (2016), no. 1, 83--114.

\bibitem{LYZ18} J. Li,  T. Yu,  and T. Zeng, \textit{Dynamics on sensitive and equicontinuous functions}, Topol. Methods Nonlinear Anal. \textbf{51} (2018), no. 2, 545--563.

\bibitem{LZ2017} K. Liu and X. Zhou, \textit{Auslander-Yorke type dichotomy theorems for stronger versions of $r$-sensitivity}, Proc. Amer. Math. Soc. \textbf{147} (2019), no. 6, 2609--2617.

\bibitem{M78} D. McMahon,  \textit{Relativized weak disjointness and relatively invariant measures}, Trans. Amer. Math. Soc. \textbf{236} (1978), 225--237.

\bibitem{Moothathu2007} T.~K.~Subrahmonian Moothathu, \textit{Stronger forms of sensitivity for dynamical systems},
Nonlinearity \textbf{20} (2007), no. 9, 2115--2126.

\bibitem{P10} F. Polo,  
\textit{Sensitive dependence on initial conditions and chaotic group actions},
Proc. Amer. Math. Soc. \textbf{138} (2010), no. 8, 2815--2826.

\bibitem{Ruelle1977} D. Ruelle, \textit{Dynamical systems with turbulent behavior}, Mathematical problems in theoretical
physics (Proc. Internat. Conf., Univ. Rome, Rome, 1977), Lecture Notes in Phys., vol. 80,
Springer, Berlin-New York, 1978, pp. 341--360.

\bibitem{SYZ} S. Shao, X. Ye and R. Zhang, \textit{Sensitivity and regionally proximal relation in minimal systems},
 Sci. China Ser. A, \textbf{51} (2008), no. 6, 987--994.

\bibitem{V68} W. Veech, \textit{The equicontinuous structure relation for minimal Abelian transformation groups}, Amer. J. Math. \textbf{90} (1968), 723--732.

\bibitem{V75} W. Veech, \textit{Topological dynamics}, Bull. Amer. Math. Soc. \textbf{83} (1977), no. 5, 775--830.

\bibitem{X05}
J. Xiong, \textit{Chaos in a topologically transitive system}, Sci. China Ser. A  \textbf{48} (2005), no. 7, 929--939.

\bibitem{Y18} X. Ye and T. Yu,
\textit{Sensitivity, proximal extension and higher order almost automorphy},
Trans. Amer. Math. Soc.  \textbf{370} (2018),  no. 5, 3639--3662.

\bibitem{YZ} X. Ye and R. Zhang, \textit{On sensitive sets in topological
dynamics}, Nonlinearity, \textbf{21} (2008), no. 7, 1601--1620.

\bibitem{YZ20} T. Yu and X. Zhou, \textit{Relativization of sensitivity in minimal systems}, Ergodic Theory Dynam. Systems 
\textbf{40} (2020), no. 6, 1715--1728.

\bibitem{Z07} G. Zhang, \textit{Relativization of complexity and sensitivity}, Ergodic Theory Dynam. Systems \textbf{27} (2007), no. 4, 1349--1371.

\bibitem{Zou17} Y. Zou,
\textit{Stronger version sensitivity, almost finite to one extension and maximal pattern
	entropy},
Commun. Math. Stat. \textbf{5} (2017), no. 2, 123--139

\end{thebibliography}
\end{document}